\documentclass[11pt,reqno]{amsart}
\usepackage{amsmath,amssymb,amsthm,ifthen,color,amsaddr, appendix}
\usepackage[all]{xypic}
\usepackage[mathcal]{euscript}
\usepackage{mathrsfs}

\newtheorem{thm}{Theorem}[section]
\newtheorem{corol}[thm]{Corollary}
\newtheorem{lemma}[thm]{Lemma}
\newtheorem{prop}[thm]{Proposition}
\newtheorem{defin}[thm]{Definition}

\theoremstyle{remark}
\newtheorem{rem}[thm]{Remark}
\newtheorem{ex}[thm]{Example}
\newenvironment{remark}{\begin{rem}\rm}{\qee\end{rem}}

\newcommand{\cO}{{\mathcal O}}

\newcommand{\rk}{{\mbox{rk}\,}}
\newcommand{\PP}{{\mathbb P}}

\newcommand{\R}{{\mathbb R}}
\newcommand{\C}{{\mathbb C}}
\newcommand{\Z}{{\mathbb Z}}

\newcommand{\Q}{{\mathbb Q}}

\newcommand{\qee}{\mbox{\hspace{0.2mm}}\hfill$\triangle$}

\newcommand{\var}{{\PP_\Sigma}}
 
  \renewcommand{\rightmark}{\sc On the Hodge conjecture  for  hypersurfaces in toric varieties}
\begin{document} 
\begin{flushright} SISSA Preprint 39/2017/MATE
\end{flushright}

\bigskip\bigskip
\title[Oda  and Hodge ]{On the Hodge conjecture  for  \\[5pt]  hypersurfaces in toric varieties}
\author{\small Ugo Bruzzo$^1$ and Antonella Grassi$^2$}
\address{\small $^1$ (SISSA) Scuola Internazionale Superiore di Studi Avanzati,  \\ Via Bonomea 265, 34136 Trieste, Italia\\
  Istituto Nazionale di Fisica Nucleare, Sezione di Trieste}
\address{\small $^2$ Department of Mathematics, University of Pennsylvania,  \\ 209 S 33rd St., Philadelphia, PA 19104, USA }
\thanks{E-mail: {\tt  bruzzo@sissa.it, grassi@sas.upenn.edu}}
\date{\today}

\begin{abstract} We show  that very general hypersurfaces in odd-dimensional simplicial projective toric varieties verifying a certain combinatorial property satisfy the Hodge conjecture  (these include projective spaces). This gives a connection between the Oda conjecture  and   Hodge conjecture. We also give an explicit criterion which depends on the degree  for very general hypersurfaces for the combinatorial condition to be verified.
 \end{abstract}

\thanks{Research partially supported by  PRIN ``Geometria delle variet\`a algebriche,''   GNSAGA-INdAM, and  the University of Pennsylvania Department of Mathematics Visitors Fund. U.B. is a member of the VBAC group}

\maketitle 

\section{Introduction}
The classical Noether-Lefschetz theorem states that if $Y$ is a very general surface in the linear system $\vert\mathcal O_{\mathbb P^3}(d)\vert$,
with $d\ge 4$, then the Picard number of $Y$ is 1. This was generalized (see e.g.~\cite{CMP03}) to very general, very ample hypersurfaces $Y$ in a smooth projective variety $X$;  in this case the result should be that $H^{p-1,d-p}(Y) $,
 the   cohomology of $Y$ in degree $(p-1,d-p)$ (where $d=\dim X$,
and $1\le p \le d-1$),  is the restriction of the cohomology of $X$ in the same degree --- or, equivalently,
that the primitive cohomology $PH^{p-1,d-p}(Y) $ vanishes.
However,  the condition that the hypersurface $Y$ has a big enough degree  is no longer sufficient to ensure that.
 A suitable sufficient  condition is that the natural morphism
\begin{equation}\label{surjective}  T_Y \mathcal M_{L} \otimes PH^{p,d-1-p}(Y)\to  PH^{p-1,d-p}(Y) \end{equation} 
(where $\mathcal M_{L} $ is the moduli space of hypersurfaces in $X$ in a very ample linear system $\vert L \vert$) is surjective
(one also needs some suitable vanishings).
 
This same result was proved by us when $X$ is a simplicial projective toric variety, possibly singular. In that case the fact that the morphism
\eqref{surjective} is surjective is a consequence of a combinatorial assumption on the toric variety $X$, which can be expressed in terms of
the Cox ring $S$ of $X$: the multiplication morphism
$$S_\alpha\otimes S_\beta \to S_{\alpha+\beta}$$
is  surjective whenever $\alpha$ and $\beta$ are an ample and nef class in the class group of $X$ \cite{BG2}. 
We  call ``Oda varieties" the varieties which satisfy this condition, after  Oda \cite{OdaUnpubl}. Examples of Oda varieties are given in Section \ref{Oda}. Projective spaces are the easiest examples. 


A toric variety $X$ satisfies the Hodge conjecture: every cohomology class in $H^{p,p}(X,\mathbb  Q)$ is a linear combination
of algebraic cycles. Then by the Lefschetz theorem, the Hodge conjecture holds true for a hypersurface  $Y$ in $X$
for $p < (d-1)/2$, where $d=\dim\var$, and by Poincar\'e duality, also for $p>(d-1)/2$. The question remains open for the intermediate cohomology
$p = (d-1)/2$ when $d$ is odd. However for Oda varieties, and for a very general hypersurface, this case is covered by the fact that the restriction morphism $H^{p,p}(X,\mathbb Q) \to H^{p,p}(Y,\mathbb Q)$ is surjective. Therefore we obtain:

\begin{thm} Let $\var$ be an Oda variety of odd dimension $d=2p+1$, and let $L$ be a very ample divisor in $\var$ such that
$pL+K_\var$ is nef. Then the Hodge conjecture with rational coefficients holds for the very general hypersurface in the linear system $\vert L \vert$.
\end{thm}

 This gives an explicit criterion which depends on the degree  for very general hypersurfaces to satisfy the Hodge conjecture.

We also formulate a partial converse of this results and pose some questions.

\smallskip\noindent{\bf Acknowledgments.} We thank B.~Fulton for a stimulating remark and V.~Srinivas for interesting correspondence about the Hodge conjecture.

\bigskip
\section{Hypersurfaces in toric simplicial varieties} 

We recall some basic facts about hypersurfaces in projective simplicial toric varieties and their cohomology.  We mainly follow   \cite{BaCox94} and \cite{BG1}, also in the notation. All schemes are  over the complex numbers.
\subsection{Preliminaries and notation} 
 Let $N$ be a free abelian group of rank $d$.
A rational simplicial complete $d$-dimensional fan $\Sigma$ in $N_\R=N\otimes_\Z\R$ defines a complete toric variety $\var $ of dimension $d$ 
  with only  Abelian quotient singularities. 
In particular $\var$  is an orbifold;  $\var$ is also called  {\it simplicial}.
 We assume here that $\var$ is  also projective.

Let  $\{\rho\}$ be the rays  of the fan $\Sigma$ and  $\{x_\rho\}$  the homogeneous coordinates associated with them.
Let $S = \C[x_\rho ]$ be the   {Cox ring}  of $\var$, that is  the algebra over $\C$ generated by the homogeneous coordinates $x_\rho$ .
Each ray $\rho$   and each coordinate $x_\rho$ determine a torus invariant Weil divisor $D_\rho$. A monomial $ \prod x_\rho^{a_\rho}$ determines a Weil divisor $D= \sum_\rho a_{\rho}D_\rho$. 

 $S$ is graded by the class group $ Cl(\var)$ of $\var$. Denoting by  $\beta=[D] $ the class of $D$ in $Cl(\var)$,
one has  $S =\oplus_{\beta\in Cl(\var)}S_\beta$. 

The Cox ring generalizes the coordinate ring of $\mathbb P^d$, in which case $S$ is the polynomial ring in $d+1$ variables over $\C$.

\subsection{Hypersurfaces  and their cohomology}\label{hypersurfaces}

Let $L$ be a nef divisor.
Then  $\cO_\var(L)$ is generated by its global sections \cite[Th.~1.6]{Mavlyutov-semi} and a general  hypersurface  $X \in |L|$ is { quasi-smooth}, that is, its only singularities are those inherited from $\var$ \cite[Lemma 6.6, 6.7]{MulletToric}. 
In particular,  if $\var$ is simplicial $X$ is  also an orbifold.

 We recall some facts from \cite{BaCox94} and Section 3 of \cite{BG1}. For the reader's convenience the Appendix provides the relevant  material
 taken from \cite{BG1}.
 
Let $L$ be a very ample Cartier divisor and  $X$ a general  hypersurface  $X \in |L|$. 
 The homotopy hyperplane Lefschetz theorem
implies that  $X$ is also simply connected if $\dim (\var) \geq 3$ \cite[Thm.~1.2 Part II]{GoMac88}.  The hard Lefschetz theorem holds also for projective orbifolds  \cite{SaitoKyoto, Zaf09}).

The complex cohomology of an orbifold has a pure
Hodge structure in each dimension \cite{Stee77,Tu86}:
 $H^{d-1}( \var ,\C)$ and $ H^{d-1}(X,\C)$ have then  pure Hodge structures.

Let  $i\colon X \to \var $ be  the natural  inclusion and  $i^\ast\colon H^\bullet( \var ,\C)
\to H^\bullet(X,\C)$ the associated morphism in cohomology; $i^\ast\colon H^{d-1}( \var ,\C) \to H^{d-1}(X,\C)$ is injective by Lefschetz's theorem. 
The morphism $i^\ast$ is compatible with the Hodge structures.

\begin{defin}\cite[Def.~10.9]{BaCox94} The primitive cohomology group $PH^{d-1}(X)$ is the quotient
$H^{d-1}(X,\C)/$ $i^\ast(H^{d-1} (\var ,\C))$.
\end{defin}

\noindent $PH^{d-1}(X)$  inherits
a pure Hodge structure, and one can write $$ PH^{d-1}(X) = \bigoplus_{p=0}^{d-1} PH^{p,d-1-p}(X).$$

Let ${\mathscr Z}$ be the open subscheme of $\vert L \vert$ pa\-ram\-e\-triz\-ing the quasi-smooth hypersurfaces in  $\vert L \vert$ and let  $\mathcal M_\beta$  be the coarse moduli space for the  quasi-smooth hypersurfaces in $\var$ with  divisor class ${\beta}$.

 There is an   ``Noether-Lefschetz theorem" for  odd dimensional simplicial toric varieties as  
 in \cite[Theorem 7.5.1]{CMP03}:

\begin{thm}\label{NL}  Let $\var$ be a simplicial toric variety of dimension $d=2p+1 \geq 3$. 
If the morphism
\begin{equation}\label{gamma2} \gamma_p\colon T_X \mathcal M_{\beta} \otimes PH^{p+1,d-p-2}(X)\to  PH^{p,d-p-1}(X) \end{equation}
 is surjective, then for $z$ away from a countable union of subschemes of ${\mathscr Z}$ of positive codimension one has
$$H^{p,p}(X_z,\Q) = \operatorname{im}[ i^\ast\colon H^{p,p}(\var,\Q) \to H^{2p}(X_z.\Q)].$$

\end{thm}

Recall that for any variety $Y$,   $H^{p,p}(Y,\Q) = H^{p,p}(Y,\C)\cap H^{2p}(Y,\Q)$.
 \subsection{Hypersurfaces in $\var$ and their Jacobian Rings}
Let $L$ be an  ample, hence very ample, Cartier divisor, of class $[L]=\beta$ on a simplicial toric variety $\var$ of dimension $d=2p+1$.
Let $f$ be a section of the line bundle
$\cO_\var(L)$  and $X$  be the hypersurface defined by $f$; it turns out that $f \in S_\beta$.
\begin{defin} Let $J(f) \subset S$  be the ideal  in the Cox ring generated by the derivatives of $f$.
The {Jacobian ring} $R(f)$ is defined as $R(f)=S/J(f)$; it is naturally graded by the class group $Cl(\var)$.
 \end{defin}

\begin{thm}\cite[Proposition 13.7; Theorem 10.13]{BaCox94} Let $\var$ be a simplicial toric variety of dimension $d=2p+1$ and $L$ be an  ample,  Cartier divisor, of class $[L]=\beta$
 Let  $\beta_0 = - [ K_{\var }]$ and $\beta=[L]$. Then:
\begin{enumerate}
\item   $T_X\mathcal M_{\beta}\simeq R(f)_{\beta}$.
\item  
$ PH^{p+1,d-p-2}(X) \simeq R(f)_{p\beta-\beta_0 }$.
\label{isoring}
\end{enumerate}\end{thm}


\begin{corol}\cite[Proposition 3.4]{BG1} \label{gammaisring}The morphism $\gamma_p$ in equation \eqref{gamma2} coincides with the multiplication in the ring $R(f)$:
\begin{equation}\label{gammaR} R(f)_\beta \otimes R(f)_{p\beta-\beta_0} \to R(f)_{(p+1)\beta-\beta_0}\, \end{equation}
\end{corol}
\bigskip
\section{ Hodge Conjecture on Toric varieties }

\begin{thm}\label{HT}  Let $\var$ be a simplicial toric variety of dimension $d=2p+1 \geq 3$. Assume that  morphism \begin{equation*} R(f)_\beta \otimes R(f)_{p\beta-\beta_0} \to R(f)_{(p+1)\beta-\beta_0}
\end{equation*}
in  equation  \eqref{gammaR} is surjective.

Then the Hodge conjecture holds for a very general hypersurface in the linear system $\vert L \vert$, that is  any class in $H^{p,p}(X,\Q) $ is represented by  a linear combination of algebraic cycles.
\end{thm}
\begin{proof} 
The Hodge conjecture holds for $\var$ toric varieties, that is, any class in $H^{p,p}(\var,\Q)$ is represented by a linear combination of  algebraic classes. Since \eqref{gammaR} is surjective,  \eqref{gamma2}  is  also surjective and Theorem \ref{NL} implies that   any class in $H^{p,p}(X,\Q) $ is also represented by  a linear combination of algebraic cycles.

\end{proof}

\noindent {\bf Question 1:}  How can one determine when  the morphism in equation \eqref{gamma2} is surjective?

\medskip
Or equivalently:
\medskip

\noindent {\bf Question 2:}  How can one determine when  the morphism in equation \eqref{gammaR} is surjective?

\medskip

\begin{remark}Note that the morphism of equation \eqref{gammaR}
\begin{equation*} R(f)_\beta \otimes R(f)_{p\beta-\beta_0} \to R(f)_{(p+1)\beta-\beta_0}
\end{equation*}
is surjective whenever the morphism
\begin{equation*} S_\beta \otimes S_{p\beta-\beta_0} \to S_{(p+1)\beta-\beta_0}\end{equation*}
is surjective.
\end{remark}
\bigskip

\section{Oda varieties}\label{Oda}

\begin{defin}\label{A1}  A toric variety $\var$ is an Oda variety if the multiplication morphism $S_{\alpha_1 }\otimes S_{\alpha_2} \to S_{\alpha_1+\alpha_2}$
is surjective whenever the classes $\alpha_1$ and $\alpha_2$ in $\operatorname{Pic}(\var)$ are
ample and nef, respectively.
\end{defin}

The  question of the surjectivity of this map was  posed by  Oda in \cite{OdaUnpubl}  under more general  conditions.
This property  can be stated in terms of the Minkowski sum of polytopes, as the  integral points of a polytope associated with a line bundle correspond to   sections of the line bundle. Definition \ref{A1} says  that  the sum $P_{\alpha_1}+ P_{\alpha_2}$  of the polytopes associated with the line bundles ${\mathcal O}_{\var}(\alpha_1)$ and ${\mathcal O}_{\var}(\alpha_2$) is equal to  their Minkowski sum, that is $P_{\alpha_1+\alpha_2}$, the polytope associated with the line bundle ${\mathcal O}_{\var}(\alpha_1+\alpha_2)$. 

The Oda conjecture is open, even for smooth varieties. The projective spaces are Oda varieties 

Some results by Ikeda  can be rephrased as follows.
\begin{thm} \label{Ikeda}\cite[Corollary ~ 4.2]{Ikeda09}
\begin{enumerate} 
\item A smooth toric variety  with Picard number $2$ is an Oda variety.
\item  The total space of a toric projective bundle over an Oda variety is also an Oda variety. 
\end{enumerate}
\end{thm}
 In \cite{BG2} we prove:
\begin{prop}\label{corA1}
Let $\var$ be a projective toric variety. If $\operatorname{Pic}(\var) = \mathbb Z$  and its ample generator $\eta$
is  Castelnuovo-Mumford 0-regular, then $\var$ is an Oda variety.
\end{prop} 



\bigskip
\section{Oda and Effective Hodge }

We then have an effective Hodge result for very general hypersurfaces in toric simplicial Oda varieties:

\begin{thm} 
Let $\var$ be an Oda variety of dimension  $d=2p+1\ge 3$ and $X \in |L|$,  a very general element, with $L$ very ample of class $[L]=\beta$, such that $p\beta-\beta_0$ is nef. 
 Then the Hodge conjecture holds for the very general hypersurface in the linear system $\vert L \vert$.
\end{thm}

\begin{corol}
\begin{enumerate}
\item (Smooth varieties) If $\var$ is smooth Fano or quasi-Fano, with Picard number $2$, then the Hodge conjecture holds for $X$ very general in  any very ample class $\beta$, $\var$, by Ikeda's  Theorem \ref{Ikeda}.
\item (Singular varieties)  In particular, if $\var=\mathbb P [1,1,2,2, \cdots,2]$ then the Hodge conjecture holds for $X$ very general in  any very ample class $\beta=k \eta$, $k \geq 3$, where $\eta$ is the very ample generator of the Picard group.
\end{enumerate}
\end{corol}

\bigskip
\section{Hodge and Oda}

 When $p=1$, and $d=3$,  Theorem \ref{HT} gives a Noether-Lefschetz theorem on the Neron-Severi group of $X$.

In the appendix we prove:

\begin{prop}\label{RS}   Let $p=1$, and $d=3$, and $\rk (Cl(X) \otimes \Q)=\rk (Cl(\var) \otimes \Q)$.  Then the morphism \eqref{gammaR}
\begin{equation*}R(f)_\beta \otimes R(f)_{\beta-\beta_0} \to R(f)_{2\beta-\beta_0}\, \end{equation*}
 is surjective.
\end{prop}
\begin{proof} See Proposition \ref{partialconv} in the Appendix.
\end{proof}
Proposition \ref{RS} combined with a result of Ravindra and Srinivas gives:
\begin{corol}  Let $p=1$,  $d=3$. Assume that $K_\var +L$  of class $\beta- \beta_0$ is  generated by its global section.  Then the morphism 
\begin{equation*} R(f)_\beta \otimes R(f)_{\beta-\beta_0} \to R(f)_{2\beta-\beta_0}
\end{equation*}
is surjective, for a very general $f \in S_\beta$\end{corol}
\begin{proof} \cite[Thm.\ 1]{RavSri09}  and \cite[Thm.\ 10]{bgl}.
\end{proof}

\begin{remark} The results above bring to the following:
\begin{itemize}
\item  Question (Noether-Lefschetz and Oda): Let $d=3$. If  the Noether-Lefschetz  theorem holds, that is, if the morphism 
\begin{equation*} R(f)_\beta \otimes R(f)_{\beta-\beta_0} \to R(f)_{2\beta-\beta_0}
\end{equation*}
is surjective, then for  a very general $f \in S_\beta$,  is the morphism
\begin{equation*} S_\beta \otimes S_{\beta-\beta_0} \to S_{2\beta-\beta_0}\end{equation*}
surjective?
\item Question (Hodge and Oda):  Let $d=2p+1$. If the Hodge conjecture holds, for very general $X \subset \var$, that is, if the morphism 
\begin{equation*} R(f)_\beta \otimes R(f)_{p\beta-\beta_0} \to R(f)_{(p+1)\beta-\beta_0}
\end{equation*}
is surjective, for a very general $f \in S_\beta$,  is the morphism
\begin{equation*} S_\beta \otimes S_{p\beta-\beta_0} \to S_{(p+1)\beta-\beta_0}\end{equation*}
also surjective?
\end{itemize}
\end{remark}
\bigskip

\appendix
\section*{Appendix}
\renewcommand{\thesection}{A}

Let $L$ be an ample, hence a very ample Cartier divisor and  $X \in |L|$ is a quasi-smooth general surface.
 
 The homotopy hyperplane Lefschetz theorem
implies that  $X$ is also simply connected if $\dim (\var) \geq 3$ \cite[Thm.~1.2 Part II]{GoMac88}.  The hard Lefschetz theorem holds also for projective orbifolds  \cite{SaitoKyoto, Zaf09}).

Let  $i\colon X \to \var $ be  the natural  inclusion and  $i^\ast\colon H^\bullet( \var ,\C)
\to H^\bullet(X,\C)$ the associated morphism in cohomology. 
Note that $i^\ast\colon H^{d-1}( \var ,\C) \to H^{d-1}(X,\C)$ is injective by Lefschetz's theorem. 

\begin{defin}\cite[Def.~10.9]{BaCox94} The primitive cohomology group $PH^{d-1}(X)$ is the quotient
$H^{d-1}(X,\C)/$ $i^\ast(H^{d-1} (\var ,\C))$.
\end{defin}
The primitive cohomology classes  $PH^{d-1}(X)$ can be represented by  differential forms of top degree on $\var$ with poles along $X$; for every $p$ with $ 0 \le p \le d-1$ there is  a naturally defined residue map
\begin{equation} \label{residue}
r_p \colon H^0(\var,\Omega^d_\var((d-p+1)X)) \to PH^{p,d-p-1}(X)\,.
\end{equation}

Let ${\mathscr Z}$ be the open subscheme of $\vert L \vert$ pa\-ram\-e\-triz\-ing the quasi-smooth hypersurfaces in  $\vert L \vert$, and let $\pi\colon \mathscr X \to {\mathscr Z}$ be the tautological family on ${\mathscr Z}$; we denote by $X_z$  the fiber of $\mathscr X$  at  $z\in \mathscr Z$.
Let $\mathscr H^{d-1}$ be the local system on ${\mathscr Z}$ whose fiber at $z$
is the   cohomology $H^{d-1}(X_z, \C)$, i.e., $\mathscr H^{d-1}=R^{d-1}\pi_\ast\C$.
It defines a flat connection $\nabla$ in the
vector bundle $\mathscr E^{d-1} = \mathscr H^{d-1} \otimes_\C \cO_{\mathscr Z}$, the
{\em Gauss-Manin connection} of $\mathscr E^{d-1}$. Since the hypersurfaces $X_z$ are quasi-smooth, the Hodge structure of the fibres $H^{d-1}(X_z, \C)$ of $\mathscr E^{d-1}$ varies analytically with $z$ \cite{Stee77}. The corresponding filtration defines holomorphic subbundles $F^p\mathscr E^{d-1}$, and the graded object of the filtration defines holomophic bundles $Gr_F^p(\mathscr E^{d-1})$. The bundles
$\mathscr E^{p,d-p-1}$ given by the Hodge decomposition are not holomorphic subbundles of $\mathscr E^{d-1}$, but are diffeomorphic to $Gr_F^p(\mathscr E^{d-1})$, and as such they have a holomorphic structure. The quotient bundles
$\mathscr {PE}^{p,d-p-1}$ of $\mathscr E^{p,d-p-1}$ correspond to the primitive cohomologies of the hypersurfaces $X_z$.
Let $\pi_p: \mathscr E^{d-1}\to \mathscr{PE}^{p,d-p+1}$ be the natural projection.

 We denote by $\tilde\gamma_p$  the cup product
$$ \tilde\gamma_p\colon H^0(\var, \cO_\var(X))   \otimes H^0(\var,\Omega_\var^d((d-p)X))
\to H^0(\var,\Omega_\var^d((d-p+1)X))\,.$$
If $z_0$ is the point in ${\mathscr Z}$ corresponding to $X$, the space $  H^0(\var, \cO_\var(X))/\C(f)$, where $\C(f)$ is the 1-dimensional subspace of $H^0(\var, \cO_\var(X))$ generated by $f$, can be identified with $T_{z_0}{\mathscr Z}$.

 The morphism $\tilde\gamma_p$ induces in cohomology   the Gauss-Manin connection:

\begin{lemma} Let $\sigma_0$ be a primitive class in $PH^{p,d-p-1}(X)$,
let $v\in T_{z_0}{\mathscr Z}$, and let $\sigma$ be a section of $\mathscr E^{p,d-p-1}$
along a curve in $\mathscr Z$ whose tangent vector at $z_0$ is $v$, such that $\sigma(z_0)=\sigma_0$.

Then   \begin{equation}\label{GM} \pi_{p-1}( \nabla_v(\sigma)) =r_{p-1}( \tilde\gamma_p(\tilde v\otimes \tilde\sigma))
\end{equation}
where  $r_p$, $r_{p-1}$ are the residue morphisms defined in equation \eqref{residue}, $\tilde\sigma$ is an element in $H^0(\var,\Omega_\var^d((d-p+1)X))$
such that $r_p(\tilde\sigma)=\sigma_0$, and $\tilde v$ is a pre-image of $v$ in $H^0(\var, \cO_\var(X))$.

In particular the following diagram commutes:
\begin{equation}\label{commuta}
\xymatrix{
\displaystyle H^0(\var, \cO_\var(X)) 
  \otimes H^0(\var,\Omega_\var^d((d-p)X))
\ar[r]^{\ \ \ \ \ \ \ \ \ \tilde\gamma_p}\ar[d]_{\phi\otimes r_p} &  H^0(\var,\Omega_\var^d((d-p+1)X))
\ar[d]_{r_{p-1}} \\
T_{z_0}{\mathscr Z} \otimes PH^{p,d-1-p}(X) \ar[r]^{\ \ \ \ \ \ \  \gamma_p }& PH^{p-1,d-p}(X) }
\end{equation}
where $\gamma_p$ is the morphism that maps $v\otimes \alpha$ to $\nabla_v\alpha$ and $\phi$ is  the projection
 $\phi\colon$ $  H^0(\var, \cO_\var(X)) \to T_{z_0}{\mathscr Z}$.
\end{lemma}

\begin{lemma} If $\alpha$ and $\eta$ are sections of $\mathscr E^{p,d-p-1}$ and
$\mathscr E^{d-p,p-1}$ respectively, then for every tangent vector $v\in T_{z_0}{\mathscr Z}$,
\begin{equation}\label{compa}\nabla_v\alpha\cup \eta= -   \alpha \cup \nabla_v\eta\,.
\end{equation}
\end{lemma}

Let ${\operatorname{Aut}}_\beta(\var)$ be the subgroup of $ \operatorname{Aut}(\var) $
which preserves the grading $\beta$. The coarse moduli space $\mathcal M_\beta$ for the general quasi-smooth hypersurfaces in $\var$ with  divisor class ${\beta}$  may be constructed as a quotient
\begin{equation}\label{moduli} U/\widetilde{\operatorname{Aut}}_\beta(\var)\,,
\end{equation} \cite{BaCox94, AGM}, where $U$ is an open subset of
$H^0(\var,\cO_{\var}(X))$, and $\widetilde{\operatorname{Aut}}_\beta(\var)$ is the unique nontrivial extension
$$ 1 \to D(\Sigma) \to \widetilde{\operatorname{Aut}}_\beta(\var) \to {\operatorname{Aut}}_\beta(\var) \to 1\,. $$
By differentiating, we have    a surjective map
$$\kappa_\beta\colon H^0(\var, \cO_\var(X)) \to T_X \mathcal M_{\beta}\,,$$ which is an analogue of the  Kodaira-Spencer map.

The  local system $\mathscr H^{d-1}$
and its various sub-systems do not descend to the moduli space $\mathcal M_{\beta}$, because the group ${\operatorname{Aut}}_\beta(\var)$ is not connected.  Nevertheless, this group
has a connected subgroup
$\operatorname{Aut}^0_\beta(\var)$ of finite order, and, perhaps after suitably shrinking $U$, the quotient $\mathcal M^0_\beta \stackrel{def}{=} U/\operatorname{Aut}^0_\beta(\var)$ is a finite \'etale covering
 of $\mathcal M_\beta$ \cite{CoxDon, AGM}. Since we are only interested in the tangent space $T_X \mathcal M_{\beta}$,
 we can replace $\mathcal M_{\beta}$ with  $\mathcal M^0_\beta$.

\begin{prop} There is a morphism
\begin{equation} \gamma_p\colon T_X \mathcal M_{\beta} \otimes PH^{p,d-1-p}(X)\to  PH^{p-1,d-p}(X) \end{equation}
such that the   diagram
\begin{equation*}\label{commuta2}\xymatrix{
H^0(\var, \cO_\var(X))
\otimes H^0(\var,\Omega_\var^d((d-p)X))
\ar[r]^{\ \ \ \ \ \ \ \ \ \cup }\ar[d]_{\kappa_\beta\otimes r_p} &  H^0(\var,\Omega_\var^d((d-p+1)X))
\ar[d]_{r_{p-1}} \\
T_X \mathcal M_{\beta} \otimes PH^{p,d-1-p}(X) \ar[r]^{\ \ \ \ \ \ \  \gamma_p }& PH^{p-1,d-p}(X) }
\end{equation*}
commutes.
\end{prop}

 Denote by $H^{d-1}_T(X) \subset H^{d-1}(X)$ the subspace
of  the cohomology classes that are annihilated by the action of the Gauss-Manin connection. Coefficients may be taken in $\C$ or $\Q$. Note that $H^{d-1}_T(X)$ has a Hodge structure.

The  result below is an  ``infinitesimal Noether-Lefschetz theorem", such as  Theorem 7.5.1 in \cite{CMP03}.

\begin{prop}\label{partialconv}For a given $p$ with $ 1 \le p \le d-1$,   
the morphism 
\begin{equation}\label{surj} \gamma_p\colon T_X \mathcal M_{\beta} \otimes PH^{p,d-1-p}(X)\to  PH^{p-1,d-p}(X)  \end{equation}
 is surjective if and only if  $H^{p,d-1-p}_T(X)=i^\ast(H^{p,d-1-p}(\var))$.
\end{prop}
\begin{proof} The ``only if'' part was proved in \cite{BG1}. To prove the ``if'' part, 
assume $\gamma_p$ is not surjective, and 
decompose $PH^{p-1,d-p}(X)$ as
$$ PH^{p-1,d-p}(X) =  \operatorname{Im} \gamma_p \oplus ( \operatorname{Im}  \gamma_p)^\perp .$$
Le $\{\eta_i\}$ be a basis of  $PH^{p,d-1-p}(X)$, and $\{t_j\}$ a basis of $T_X \mathcal M_{\beta}$.  Fix values for $i$ and $j$ and let
$\tau = \gamma_p(t_j\otimes\eta_i)$. If $\alpha\in  ( \operatorname{Im}  \gamma_p)^\perp$, then
$$ 0 = \langle \alpha,\tau\rangle  = \langle\alpha, \gamma_p(t_j\otimes\eta_i)\rangle = \langle\nabla_{t_j}\alpha,\eta_i\rangle \quad\mbox{for all} \ i $$
so that $\nabla_{t_j}\alpha=0$  for all $j$, i.e., $\alpha$ is a nonzero element in $H^{p-1,d-p}_T(X)$, which implies that
$H^{p,d-1-p}_T(X)$ is properly contained in $ i^\ast(H^{p,d-1-p}(\var))$.
\end{proof}

\begin{lemma} Let $d=2p+1\ge 3$, and assume that 
 the  hypotheses of the previous Lemma hold for $p=m$. Then for $z$ away from a countable union of subschemes of ${\mathscr Z}$ of positive codimension one has
$$H^{p,p}(X_z,\Q) = \operatorname{im}[ i^\ast\colon H^{p,p}(\var,\Q) \to H^{2p}(X_z.\Q)].$$
\end{lemma}
\begin{proof}
Let $\bar {\mathscr Z}$ be the universal cover of ${\mathscr Z}$. On it the (pullback of the) local system
$\mathscr H^{d-1}$ is trivial. Given a class $\alpha\in H^{p,p}(X)$
we can extend it to a global section of $\mathscr H^{d-1}$ by parallel transport using
the Gauss-Manin connection. Define the subset $\bar {\mathscr Z}_\alpha$ of $\bar {\mathscr Z}$ as the common zero locus of the sections $\pi_m(\alpha)$ of $\mathscr E^{m,d-1-m}$ for $p\ne m$ (i.e., the locus where $\alpha$ is of type $(p,p)$).

If $\bar {\mathscr Z}_\alpha=\bar {\mathscr Z}$ we are done because $\alpha$ is in
$H^{d-1}_T(X)$ hence it is in the image of $i^\ast$ by the previous Lemma.
If $\bar {\mathscr Z}_\alpha \ne \bar {\mathscr Z}$, we note that $\bar {\mathscr Z}_\alpha$ is a subscheme of  $\bar {\mathscr Z}$.

We subtract from ${\mathscr Z}$ the union of  the projections of the
 subschemes $\bar {\mathscr Z}_\alpha$  where $\bar {\mathscr Z}_\alpha \ne \bar {\mathscr Z}$. The set of these varieties is countable because we are considering rational classes.
\end{proof}

\bigskip\frenchspacing
\def\cprime{$'$}

%
%
%

\end{document}